\documentclass[ECP,preprint]{ejpecp}
\usepackage[T1]{fontenc}
\usepackage{tikz}
\usetikzlibrary{arrows.meta,positioning,calc}
\definecolor{SKBlue}{HTML}{285674}
\definecolor{SKGreen}{HTML}{396551}
\hypersetup{hidelinks,
  pdftitle={Disorder transfer for the critical SK overlap law},
  pdfauthor={Yan Ru Pei}}

\newcommand{\E}{\mathbb E}
\newcommand{\R}{\mathbb R}
\newcommand{\Law}{\operatorname{Law}}
\newcommand{\calP}{\mathcal P}
\newcommand{\W}{W_2}
\newcommand{\WW}{\mathcal W_2}
\newcommand{\av}[1]{\left\langle #1\right\rangle}
\newcommand{\norm}[1]{\left\lVert #1\right\rVert}

\SHORTTITLE{Disorder transfer for the critical SK overlap law}
\TITLE{Disorder transfer for the critical SK overlap law}
\AUTHORS{Yan Ru Pei\footnote{\EMAIL{yanrpei@gmail.com}}}
\AMSSUBJ{82B44}
\AMSSUBJSECONDARY{60K35}

\ABSTRACT{%
We give a quantitative disorder comparison for the zero-field
Sherrington--Kirkpatrick model at inverse temperature one. For any fixed
number of replicas, replacing Gaussian couplings by independent couplings
from a bounded symmetric variance-one law changes a bounded replica expectation
by at most a constant times the Gaussian overlap second moment plus
$N^{-1}$. A multiplicative comparison for nonnegative observables controls
the tails. The proof uses Talagrand's positive-observable estimates and
Yu-Ting Chen's fourth-order replica cancellation. Assuming the Gaussian
critical overlap limit stated by Du and Huang, we obtain the same limit
for the random quenched overlap measure in the iterated Wasserstein-2
topology. No separate moment assumption is needed for this transfer.
Their additional exponential moment bound gives convergence of the
rescaled spin-glass susceptibility in every finite Wasserstein distance.
The coupling class includes fair signs, uniform disorder, and signs
thinned at any fixed positive density.}

\begin{document}

\section{The comparison and its critical application}

At the critical temperature of the Sherrington--Kirkpatrick (SK) model, the
overlap of two independent Gibbs samples fluctuates on a scale that vanishes
with the system size. Du and Huang~\cite{DHOverlap} state that, for Gaussian
couplings, the quenched law of $N^{1/3}R_{12}$ converges to an explicit random
probability measure. This limit, built from the reflected Airy$_1$ point
process and a pinned Gaussian construction, describes the overlap
distribution within a disorder sample. We show that their Gaussian result
transfers to independent fair signs and, more generally, to bounded
symmetric couplings of variance one.

The comparison follows from existing Gibbs-measure techniques. Talagrand's
local estimates~\cite[Lemmas~2.6--2.7]{Talagrand} retain the expectation of
a nonnegative observable when bounding its coupling derivatives.
Yu-Ting Chen~\cite[Eq.~(3.19), Proposition~4.2 and Example~4.3]{Chen}
reduces the summed fourth-order derivative to squared two- and four-replica
overlaps. We combine these ingredients in a mixture interpolation and
record their quantitative consequence. Bounded tests identify the limiting
random measure; positive tests control the error from clipping its tails.
The critical conclusion is conditional on the Gaussian random-measure
convergence stated below.

General comparisons for Gibbs observables were developed by Auffinger and
Wei-Kuo Chen~\cite[author's preprint, Theorem~3]{AC}. Their direct bound
under matching of the first three moments has order one at the SK
normalization. The small error here uses the overlap cancellation in
Yu-Ting Chen's argument, which already permits $N$-dependent bounded
replica tests. Critical free-energy universality is treated by Cheng,
Liu, Shao and Xu~\cite[Theorem~2.4]{CLSX}; our application concerns the
random quenched overlap measure.

For $N\ge2$, let $\Sigma_N=\{-1,1\}^N$ and set
\begin{equation}\label{eq:model}
 H_N^J(\sigma)=\frac1{\sqrt N}\sum_{1\le i<j\le N}
 J_{ij}\sigma_i\sigma_j,
 \qquad G_N^J(\sigma)=\frac{e^{H_N^J(\sigma)}}
 {\sum_{\tau\in\Sigma_N}e^{H_N^J(\tau)}}.
\end{equation}
Thus the inverse temperature is one and the external field is zero.
Angle brackets denote expectation over independent replicas
$\sigma^1,\sigma^2,\ldots$ sampled from $G_N^J$, conditional on $J$.
Write $R_{\ell\ell'}=N^{-1}\sum_{i=1}^N\sigma_i^\ell\sigma_i^{\ell'}$.
Let $\nu$ be a symmetric probability law supported on $[-b,b]$, where
$b\ge1$, with $\int x^2\,\nu(dx)=1$. The notation $\E_G$ refers to
independent standard Gaussian couplings and $\E_\nu$ to independent
couplings with common law $\nu$.

\begin{theorem}[Finite-volume comparison]\label{thm:comparison}
For every integer $r\ge1$ and $b\ge1$ there are finite constants
$C_{r,b}$ and $L_{r,b}$ such that, for every law $\nu$ as above,
every $N\ge2$ and every deterministic observable $A:\Sigma_N^r\to\R$,
\begin{equation}\label{eq:comparison}
 \left|\E_\nu\av{A}-\E_G\av{A}\right|
 \le C_{r,b}\norm{A}_\infty
 \left(\E_G\av{R_{12}^2}+\frac1N\right).
\end{equation}
For every nonnegative $A:\Sigma_N^r\to[0,\infty)$, we also have
\begin{equation}\label{eq:positive}
 e^{-L_{r,b}}\E_G\av{A}\le\E_\nu\av{A}
 \le e^{L_{r,b}}\E_G\av{A}.
\end{equation}
The observable may depend on $N$, and its bound need not be uniform in $N$.
\end{theorem}

Examples include fair signs, uniform couplings on $[-\sqrt3,\sqrt3]$,
and $X=\varepsilon\eta/\sqrt p$, where $\varepsilon$ is a fair sign,
$\eta$ is an independent Bernoulli($p$) variable and $p\in(0,1]$ is fixed.
In the last example, $b=p^{-1/2}$, so the constants are not uniform as
$p\downarrow0$.

The estimate compares disorder-averaged replica expectations. In particular,
it applies directly to bounded tests of critically rescaled overlaps: no
derivatives of those test functions are required. The coupling laws agree in their first three moments; their fourth
moments may differ. The fourth derivative must therefore be summed using its
overlap structure; a separate absolute bound for each edge would give an
error of order one. Symmetry cancels the fifth-order term. The sixth-order
Taylor remainder and the centered restoration of one coupling each
contribute $O(N^{-1})$ after summing over the edges.

To state the application, define the random quenched measure
\begin{equation}\label{eq:mu}
 \mu_N^J=\av{\delta_{N^{1/3}R_{12}}},
 \qquad \zeta_N^G=\Law_G(\mu_N^J),
 \qquad \zeta_N^\nu=\Law_\nu(\mu_N^J).
\end{equation}
We equip $\calP_2(\R)$ with its usual Wasserstein distance $\W$ and equip
$\calP_2(\calP_2(\R))$ with the Wasserstein distance $\WW$ whose underlying
distance is $\W$. Thus
\[
 \WW(\zeta,\zeta')^2
 =\inf_{\pi\in\Pi(\zeta,\zeta')}
   \int\W(\mu,\mu')^2\,\pi(d\mu,d\mu').
\]
These are laws of random probability measures, so this topology retains
the disorder-dependent overlap distribution. We call it the iterated
Wasserstein-2 topology.

\begin{corollary}[Critical limit transfer]\label{cor:critical}
Suppose that, for Gaussian disorder in~\eqref{eq:model},
\begin{equation}\label{eq:inputs}
 \WW(\zeta_N^G,\zeta_*)\longrightarrow0
\end{equation}
for some $\zeta_*\in\calP_2(\calP_2(\R))$.
Then, for every coupling law $\nu$ in Theorem~\ref{thm:comparison},
\begin{equation}\label{eq:critical}
 \WW(\zeta_N^\nu,\zeta_*)\longrightarrow0.
\end{equation}
In particular, conditional on the Gaussian assertion in
Du--Huang~\cite[Theorem~1.1(a)]{DHOverlap}, the limit is
$\zeta_*=\Law(\mathfrak P_{a(\chi)})$, with
$\mathfrak P_{a(\chi)}$ defined in their Definitions~1.5 and~1.7.
\end{corollary}

The GOE convention in the Gaussian paper agrees
with~\eqref{eq:model}: in $H_N(\sigma)=\tfrac12(\sigma,W\sigma)$,
the off-diagonal entries have variance $1/N$, while the diagonal term
$\tfrac12\sum_iW_{ii}\sigma_i^2$ is independent of $\sigma$ and cancels
from the Gibbs measure.

Assumption~\eqref{eq:inputs} bounds the Gaussian expectation of
$\int x^2\,\mu_N^J(dx)=N^{2/3}\av{R_{12}^2}$, so the finite-volume
estimate is
$O_{r,b}(\norm{A}_\infty N^{-2/3})$ for each fixed $r$.
The random-measure convergence in~\eqref{eq:critical} is qualitative.
Section~\ref{sec:susceptibility} derives the limiting law and moments
of the critical spin-glass susceptibility. Figure~\ref{fig:transfer} shows why the proof
uses both parts of Theorem~\ref{thm:comparison}.

\begin{figure}[tbp]
\centering
\resizebox{\linewidth}{!}{%
\begin{tikzpicture}[font=\small,>=Stealth,
  box/.style={draw,rounded corners=2pt,align=center,text width=6.2cm,
    minimum height=1.4cm,inner sep=7pt},
  flow/.style={->,thick}]
  \node[box,draw=SKBlue] (bounded) at (0,0)
    {\textbf{Bounded replica tests}\\[3pt]
     comparison error $\le C_{r,b}\|A\|_\infty N^{-2/3}$};
  \node[box,draw=SKGreen] (positive) at (7.4,0)
    {\textbf{Positive replica tests}\\[3pt]
     $\E_\nu\av A\le e^{L_{r,b}}\E_G\av A$};
  \node[box,draw=SKBlue] (identify) at (0,-2.05)
    {Clipped quenched measures\\[3pt]
     $\E\prod_j\int\varphi_j\,d\mu_N^{J,L}$};
  \node[box,draw=SKGreen] (tails) at (7.4,-2.05)
    {Quadratic clipping error\\[3pt]
     $\E_\nu h_L(\mu_N^J)^2\le e^{L_{2,b}}\E_G h_L(\mu_N^J)^2$};
  \draw[flow,SKBlue] (bounded)--node[right,font=\footnotesize]
    {disjoint replica pairs}(identify);
  \draw[flow,SKGreen] (positive)--node[right,font=\footnotesize]
    {$A=(|N^{1/3}R_{12}|-L)_+^2$}(tails);
  \node[box,text width=9.7cm] (limit) at (3.7,-4.25)
    {\textbf{Same law of the random overlap measure}\\[3pt]
     $\mathcal W_2(\zeta_N^\nu,\zeta_*)\longrightarrow0$};
  \draw[flow,SKBlue] (identify.south)--node[left,font=\footnotesize]
    {identifies the limit}(identify.south |- limit.north);
  \draw[flow,SKGreen] (tails.south)--node[right,font=\footnotesize]
    {remove clipping}(tails.south |- limit.north);
\end{tikzpicture}}
\caption{The disorder transfer under the Gaussian
assumption~\eqref{eq:inputs}. Clipping overlaps to $[-L,L]$ gives
$\mu_N^{J,L}$ and the error
$h_L(\mu)^2=\int(|x|-L)_+^2\,\mu(dx)$.
For fixed $L$, bounded replica tests identify the law of the clipped
measure on a compact space. The positive comparison then allows
$L\to\infty$. The rate in the upper left applies to fixed replica
tests, not to the distance in the final box.}
\label{fig:transfer}
\end{figure}
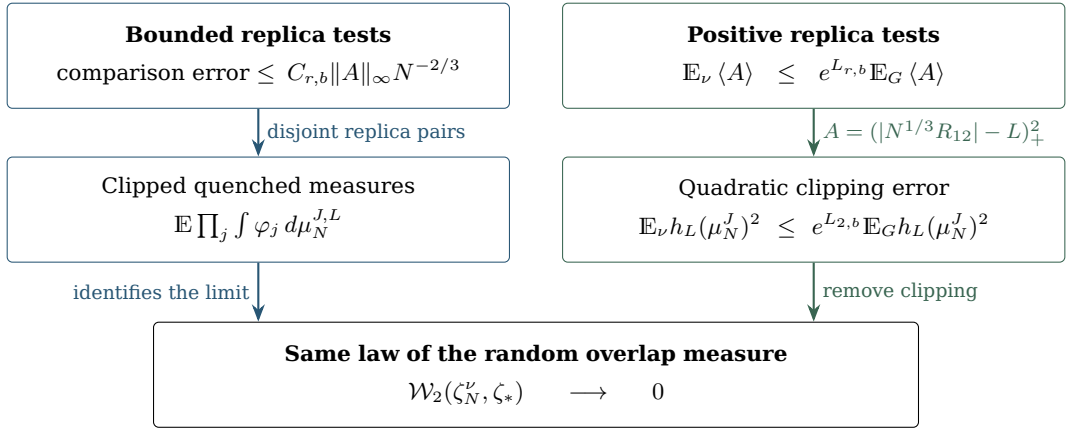

\section{Replica derivatives and positive comparison}

Let $\gamma$ be the standard Gaussian law and put
\[
 \rho_s=(1-s)\gamma+s\nu,\qquad
 F_s(A)=\E_{\rho_s}\av A,\qquad
 q_s=F_s(R_{12}^2),\qquad a=N^{-1/2},\qquad 0\le s\le1.
\]
Each coupling is sampled independently from the mixture law $\rho_s$.
For a fixed edge $e$, let $g_e(x)$ be the replica expectation with that
coupling set to $x$ and the other couplings held fixed. Differentiating
the finite product of mixture laws gives the exact identity
\begin{equation}\label{eq:mixture}
 F_s'(A)=\sum_e\E_{\rho_s}^{(e)}
 \left[\E_\nu g_e(X)-\E_\gamma g_e(G)\right],
\end{equation}
where $\E_{\rho_s}^{(e)}$ averages only the other edges. We will estimate
this single-edge difference by Taylor expansion, retaining its sign
until the fourth-order terms have been summed.

\subsection{Replica derivatives}

Fix $e=\{i,j\}$, write $b_\ell=\sigma_i^\ell\sigma_j^\ell$, and abbreviate
$g_e$ to $g$. Put $P_{r,0}=1$ and
\begin{equation}\label{eq:polynomial}
 P_{r,k}(b)=\prod_{j=0}^{k-1}
 \left(\sum_{\ell=1}^{r+j}b_\ell-(r+j)b_{r+j+1}\right),
 \qquad D_{r,k}=2^k r(r+1)\cdots(r+k-1).
\end{equation}
Differentiating a normalized replica expectation introduces one additional
replica at each step, giving
\begin{equation}\label{eq:derivatives}
 g^{(k)}(x)=a^k\av{A P_{r,k}}_x,\qquad
 \|P_{r,k}\|_\infty\le D_{r,k}.
\end{equation}
This is Chen's replica differentiation identity~\cite[Eqs.~(2.1) and~(2.6)]{Chen}
for one unordered edge. The coefficient $\ell^1$ norm of $P_{r,k}$ is
also at most $D_{r,k}$, including after reduction by $b_\ell^2=1$.
For $A\ge0$, the derivative bound retains the expectation:
\begin{equation}\label{eq:relative-derivative}
 |g^{(k)}(x)|\le a^kD_{r,k}g(x),\qquad
 g(x)\le e^{2ra|x-y|}g(y).
\end{equation}
The second inequality compares the $r$-replica Gibbs densities: the
unnormalized density ratio and its normalization each contribute at most
$e^{ra|x-y|}$. This is the local positive-observable argument of
Talagrand~\cite[Lemmas~2.6--2.7]{Talagrand}. Locators for that paper and
for~\cite{AC} refer to the author preprints listed in the bibliography.

\subsection{A bound that retains positive observables}

\begin{lemma}[Positive comparison along the mixture]\label{lem:positive}
There is $L_{r,b}<\infty$ such that, for $A\ge0$ and $0\le s\le1$,
\begin{equation}\label{eq:positive-interpolation}
 e^{-L_{r,b}s}F_0(A)\le F_s(A)\le e^{L_{r,b}s}F_0(A).
\end{equation}
\end{lemma}

\begin{proof}
Taylor expansion of $g$ through degree three cancels the matched moments.
By~\eqref{eq:relative-derivative} and $a\le1$,
\[
 |\E_\nu g(X)-\E_\gamma g(G)|
 \le\frac{D_{r,4}a^4}{24}M_{r,b}g(0),\qquad
 M_{r,b}=b^4e^{2rb}+\E_\gamma[|G|^4e^{2r|G|}]<\infty.
\]
The same density comparison gives
\[
 \E_{\rho_s}g(X)\ge m_{r,b}g(0),\qquad
 m_{r,b}=\min\{e^{-2rb},\E_\gamma e^{-2r|G|}\}>0.
\]
Since $\binom N2a^4\le1/2$, summing~\eqref{eq:mixture} yields
\[
 |F_s'(A)|\le L_{r,b}F_s(A),\qquad
 L_{r,b}=\frac{D_{r,4}M_{r,b}}{48m_{r,b}}.
\]
Gronwall's inequality proves the result, including $F_0(A)=0$.
\end{proof}

\section{Fourth-order cancellation and the error bound}

The first unmatched moment in~\eqref{eq:mixture} is the fourth.
The cancellation below is the two- and four-replica structure in
Chen~\cite[Proposition~4.2 and Example~4.3]{Chen}, written directly for
the independent unordered edges in~\eqref{eq:model}.

\begin{lemma}[Summed fourth derivative]\label{lem:cancellation}
For any deterministic $r$-replica observable $A$ and $0\le s\le1$,
\begin{equation}\label{eq:cancellation-bound}
 \left|N^{-2}\sum_e\E\av{A P_{r,4}(b_e)}_s\right|
 \le\frac{D_{r,4}}2\norm{A}_\infty q_s.
\end{equation}
\end{lemma}

\begin{proof}
After reduction by $b_\ell^2=1$, the polynomial $P_{r,4}$ has only degrees
zero, two and four. Its constant coefficient is zero for every $r$.
To see this directly, set $A=1$, remove all couplings, and evaluate the
fourth derivative in~\eqref{eq:derivatives} at $x=0$. The $b_\ell$ are then
independent fair signs, so their expectation selects the constant coefficient;
the derivative of $\av{1}_x=1$ vanishes.
Also $P_{r,4}(1,\ldots,1)=0$. Consequently
\begin{equation}\label{eq:coefficients}
 P_{r,4}(b)=\sum_{B:\,|B|\in\{2,4\}}c_B\prod_{\ell\in B}b_\ell,
 \qquad \sum_Bc_B=0,
 \qquad \sum_B|c_B|\le D_{r,4}.
\end{equation}
Here $B\subseteq\{1,\ldots,r+4\}$.
For $R_B=N^{-1}\sum_i\prod_{\ell\in B}\sigma_i^\ell$, we have exactly
\begin{equation}\label{eq:unordered}
 \sum_{i<j}\prod_{\ell\in B}(\sigma_i^\ell\sigma_j^\ell)
 =\frac{N^2R_B^2-N}{2}.
\end{equation}
The terms proportional to $N$ cancel by $\sum_Bc_B=0$.
At fixed disorder, replica independence gives
\[
 \av{R_B^2}_s
 =\frac1{N^2}\sum_{i,j}\av{\sigma_i\sigma_j}_s^{|B|}
 \le\frac1{N^2}\sum_{i,j}\av{\sigma_i\sigma_j}_s^2
 =\av{R_{12}^2}_s,
\]
because $|B|$ is two or four. Sum~\eqref{eq:unordered} with the coefficients
in~\eqref{eq:coefficients}, multiply by $A$, and take absolute values only
after this sum. The claimed bound follows.
\end{proof}

The identity~\eqref{eq:unordered} also fixes the normalization: replacing
an unordered coupling by a sum of two independent ordered couplings
would generally change its law.

\begin{proof}[Proof of Theorem~\ref{thm:comparison}]
The positive comparison is Lemma~\ref{lem:positive} at $s=1$.
For the signed estimate, set $\kappa_4=\E_\nu X^4-3$.
Symmetry cancels the odd moments through degree five. Taylor expansion
and~\eqref{eq:derivatives} therefore give, for every fixed edge environment,
\begin{equation}\label{eq:fourth-term}
 \E_\nu g(X)-\E_\gamma g(G)
 =\frac{\kappa_4}{24}g^{(4)}(0)
   +O_{r,b}(a^6\|A\|_\infty).
\end{equation}
Here the sixth moments are bounded by $b^6$ and $15$, respectively,
and the sixth derivative bound is uniform on the whole real line.
To restore the edge to its full mixture law, expand $g^{(4)}$ to first
order at zero. The mixture is centered and has variance one, so
\begin{equation}\label{eq:restore}
 |\E_{\rho_s}g^{(4)}(X)-g^{(4)}(0)|
 \le\tfrac12a^6D_{r,6}\|A\|_\infty.
\end{equation}
Equations~\eqref{eq:mixture}, \eqref{eq:fourth-term} and~\eqref{eq:restore}
place all edges under the same disorder law:
\begin{equation}\label{eq:restored}
 F_s'(A)=\frac{\kappa_4a^4}{24}
     \sum_e\E_{\rho_s}\av{A P_{r,4}(b_e)}
     +O_{r,b}(N^{-1}\|A\|_\infty).
\end{equation}
The error uses $\binom N2a^6\le(2N)^{-1}$; the fourth-cumulant
coefficient is common to all edges because their laws are identical.
Lemma~\ref{lem:cancellation} bounds the sum, and Lemma~\ref{lem:positive}
gives $q_s\le e^{L_{2,b}}q_0$. Hence
\[
 |F_s'(A)|\le\|A\|_\infty
 \left(\frac{|\kappa_4|D_{r,4}}{48}q_s
       +\frac{C'_{r,b}}{N}\right).
\]
Since $|\kappa_4|\le b^4+3$, integration over $s\in[0,1]$ proves
\eqref{eq:comparison} with a constant depending only on $r$ and $b$.
\end{proof}

\section{Convergence of the random overlap measure}
\label{sec:convergence}

We prove Corollary~\ref{cor:critical} by clipping the overlap to a compact
interval and then removing the clipping. This uses only the Gaussian
convergence~\eqref{eq:inputs}; no separate exponential moment bound is
needed. The compactness and moment principles are the same as those used
in~\cite[Section~3]{DHOverlap}.

First, the triangle inequality against $\delta_{\delta_0}$ gives
\[
 \sup_N\E_G\int x^2\,\mu_N^J(dx)<\infty,
 \qquad q_0=O(N^{-2/3}).
\]
For $L>0$, let $T_L(x)=\max\{-L,\min\{x,L\}\}$ and denote by
$\mu^L=(T_L)_\#\mu$ the law obtained by clipping a sample from $\mu$.
Write $\zeta_N^{J,L}=\Law_J((\mu_N^J)^L)$ for $J=G,\nu$, and let
$\zeta_*^L$ be the corresponding image of $\zeta_*$.
The space $\mathcal C_L=\calP([-L,L])$ is compact in $\W$, and clipping
is a $1$-Lipschitz map in that metric. Thus~\eqref{eq:inputs} implies
$\WW(\zeta_N^{G,L},\zeta_*^L)\to0$.

For continuous real functions $\varphi_1,\ldots,\varphi_m$ on $[-L,L]$,
use the bounded $2m$-replica observable
\[
 A_N=\prod_{j=1}^m
 \varphi_j\bigl(T_L(N^{1/3}R_{2j-1,2j})\bigr).
\]
Conditional independence of the disjoint replica pairs gives
\begin{equation}\label{eq:moments}
 \av{A_N}=\prod_{j=1}^m\int\varphi_j\,d(\mu_N^J)^L.
\end{equation}
Theorem~\ref{thm:comparison} makes the Gaussian and $\nu$-disorder expectations
of~\eqref{eq:moments} asymptotically equal. Finite linear combinations of
these products form an algebra of continuous functions on $\mathcal C_L$
that contains constants and separates measures. Stone--Weierstrass
therefore gives weak convergence of $\zeta_N^{\nu,L}$ to $\zeta_*^L$.
Compactness upgrades this to
\begin{equation}\label{eq:clipped-limit}
 \WW(\zeta_N^{\nu,L},\zeta_*^L)\longrightarrow0
 \qquad\text{for each fixed }L.
\end{equation}

To remove the clipping, put
\[
 h_L(\mu)=\W(\mu,\mu^L)
 =\left(\int(|x|-L)_+^2\,\mu(dx)\right)^{1/2}.
\]
The displayed equality follows from the coupling $x\mapsto T_L(x)$:
every target point in $[-L,L]$ is at least $(|x|-L)_+$ from $x$.
In particular, $h_L$ is the distance to $\mathcal C_L$, hence is
$1$-Lipschitz in $\W$. The positive comparison applied to
$A_N=(|N^{1/3}R_{12}|-L)_+^2$ yields
\begin{align}
 \WW(\zeta_N^\nu,\zeta_N^{\nu,L})
 &\le e^{L_{2,b}/2}\|h_L\|_{L^2(\zeta_N^G)}\notag\\
 &\le e^{L_{2,b}/2}
 \left(\WW(\zeta_N^G,\zeta_*)+\|h_L\|_{L^2(\zeta_*)}\right).
 \label{eq:clipping-error}
\end{align}
The second inequality follows by coupling the two random measures and
using the Lipschitz property and Minkowski's inequality. All laws here
have finite outer second moments; for the $\nu$-disorder law this also follows
from~\eqref{eq:positive} with $A_N=N^{2/3}R_{12}^2$.

Finally, $\WW(\zeta_*,\zeta_*^L)\le\|h_L\|_{L^2(\zeta_*)}$.
Combining~\eqref{eq:clipped-limit}--\eqref{eq:clipping-error} gives
\[
 \limsup_{N\to\infty}\WW(\zeta_N^\nu,\zeta_*)
 \le (e^{L_{2,b}/2}+1)\|h_L\|_{L^2(\zeta_*)}.
\]
The right side tends to zero as $L\to\infty$ by dominated convergence,
since $\zeta_*$ has a finite outer second moment. This proves the corollary.

\section{Critical spin-glass susceptibility}\label{sec:susceptibility}

The random-measure limit also controls a standard response observable.
At zero field, define the sample spin-glass susceptibility by
\begin{equation}\label{eq:susceptibility}
 \chi_{\mathrm{SG},N}^J
 =\frac1N\sum_{i,j}\av{\sigma_i\sigma_j}_J^2
 =N\av{R_{12}^2}_J.
\end{equation}
Every one-spin expectation vanishes by global spin-flip symmetry, so the
two-spin expectation here is also the connected correlation. The
rescaled susceptibility is the second moment of the quenched measure:
\[
 S_N^J=N^{-1/3}\chi_{\mathrm{SG},N}^J
      =\int x^2\,\mu_N^J(dx).
\]

\begin{corollary}[Susceptibility law and moments]\label{cor:susceptibility}
Assume~\eqref{eq:inputs}, and let $\mu_*$ have law $\zeta_*$.
Then the law of $S_N^J$ under $\nu$-disorder converges in $W_1$ to
the law of $S_*:=\int x^2\,\mu_*(dx)$.
If, in addition, there are $c>0$ and $M<\infty$ such that
\begin{equation}\label{eq:exponential-input}
 \sup_{N\ge2}\E_G\av{e^{cN^{1/3}|R_{12}|}}\le M,
\end{equation}
then, for every finite $p\ge1$,
\begin{equation}\label{eq:susceptibility-limit}
 W_p\bigl(\Law_\nu(S_N^J),\Law(S_*)\bigr)\longrightarrow0,
 \qquad \E_\nu[(S_N^J)^p]\longrightarrow\E[S_*^p]<\infty.
\end{equation}
Here $W_p$ is the usual Wasserstein distance on probability laws on $\R$.
\end{corollary}

\begin{proof}
The map $\mu\mapsto\int x^2\,\mu(dx)$ is continuous in $\W$, so
Corollary~\ref{cor:critical} gives convergence in distribution.
Its iterated Wasserstein conclusion also gives convergence of the
outer second moments, which are exactly $\E_\nu S_N^J$.
Since $S_N^J,S_*\ge0$, the Wasserstein moment criterion gives $W_1$
convergence.
Under~\eqref{eq:exponential-input}, the positive comparison yields
\begin{equation}\label{eq:transferred-moment}
 \sup_N\E_\nu\int e^{c|x|}\,\mu_N^J(dx)\le e^{L_{2,b}}M.
\end{equation}
For every real $k\ge1$, Jensen's inequality gives
\[
 \sup_N\E_\nu[(S_N^J)^k]
 \le\sup_N\E_\nu\int |x|^{2k}\,\mu_N^J(dx)<\infty.
\]
Choosing $k>p$ makes the $p$th powers uniformly integrable. Their
expectations therefore converge, and the Wasserstein moment criterion
gives~\eqref{eq:susceptibility-limit}.
\end{proof}

Du--Huang~\cite[Proposition~3.2]{DHOverlap} state
\eqref{eq:exponential-input} with $M=2$, citing
\cite[Theorem~1.4(a)]{DHFree}. A bound for $N\ge N_0$ would suffice:
the finitely many remaining sizes can be absorbed into $M$ because
$|R_{12}|\le1$. For their limiting measure,
$S_*=V_{a(\chi)}$ by~\cite[Proposition~1.6 and Definition~1.7]{DHOverlap}.
In particular,
\[
 N^{2/3}\E_\nu\av{R_{12}^2}\longrightarrow\E_\chi V_{a(\chi)}.
\]
Their Corollary~1.2 identifies this constant for Gaussian disorder and
relates it to Talagrand's Conjecture~11.7.5. The result here transfers
that conclusion, and the susceptibility law, to the bounded symmetric
coupling class.

\end{document}